\documentclass[reqno, 12pt]{amsart}

\pdfoutput=1

\usepackage{enumerate}
\usepackage{latexsym}
\usepackage[centertags]{amsmath}
\usepackage{amsfonts}
\usepackage{amssymb}
\usepackage{amsthm}
\usepackage{newlfont}
\usepackage{graphics}
\usepackage{color}
\usepackage{float}
\usepackage{diagbox}
\usepackage{longtable}
\usepackage{rotating}
\usepackage{multirow}

\usepackage{extarrows}

\usepackage[sort,compress,numbers]{natbib}

\usepackage[utf8]{inputenc}

\newtheorem{thm}{Theorem}[section]
\newtheorem{cor}[thm]{Corollary}
\newtheorem{lem}[thm]{Lemma}

\theoremstyle{mydefinition}
\newtheorem{dfn}[thm]{Definition}
\theoremstyle{myremark}
\newtheorem{rem}[thm]{Remark}
\newtheorem{exa}[thm]{Example}

\allowdisplaybreaks[4]

\renewcommand{\P}{{\mathbb{P}}}

\title{The Frobenius Formula for $A=(a,ha+d,ha+b_2d,...,ha+b_kd)$}

\author{Feihu Liu$^{1}$, Guoce Xin$^{2, *}$, Suting Ye$^{3}$ and Jingjing Yin$^{4}$}

\address{$^{1, 2, 3, 4}$School of Mathematical Sciences,  Capital Normal University,
 Beijing 100048,  PR China}
\email{$^1$\texttt{liufeihu7476@163.com}\ \& $^2$\texttt{guoce\_xin@163.com}\ \& $^3$\texttt{yesuting0203@163.com}\ \newline \newline \& $^4$\texttt{yinjingj@163.com}}
\date{April 26,  2023}
\thanks{$*$ This work was partially supported by NSFC(12071311).}
\begin{document}
\maketitle

\begin{abstract}
Given relative prime positive integers $A=(a_1, a_2, ..., a_n)$, the Frobenius number $g(A)$ is the largest integer not representable as a linear combination of the $a_i$'s with nonnegative integer coefficients. We find the ``Stable" property introduced for the square sequence $A=(a,a+1,a+2^2,\dots, a+k^2)$ naturally extends
for $A(a)=(a,ha+dB)=(a,ha+d,ha+b_2d,...,ha+b_kd)$. This gives a parallel characterization of $g(A(a))$ as a ``congruence class function" modulo $b_k$ when $a$ is large enough. For orderly sequence $B=(1,b_2,\dots,b_k)$, we find good bound for $a$. In particular we calculate $g(a,ha+dB)$ for $B=(1,2,b,b+1)$, $B=(1,2,b,b+1,2b)$, $B=(1,b,2b-1)$ and $B=(1,2,...,k,K)$. Our idea also applies to the case $B=(b_1,b_2,...,b_k)$, $b_1> 1$.
\end{abstract}

\def\D{{\mathcal{D}}}

\noindent
\begin{small}
 \emph{Mathematic subject classification}: Primary 11D07; Secondary 05A15, 11B75, 11D04.
\end{small}

\noindent
\begin{small}
\end{small}

\section{Introduction}

Given relative prime positive integers $A=(a_1, a_2, ..., a_n)$, the Frobenius number $g(A)$ is the largest integer not representable as a linear combination of the $a_i$'s with nonnegative integer coefficients. The calculation of $g(A)$ has been widely studied. See \cite{Ramrez Alfonsn} for a detailed overview. For $n=2$, Sylvester \cite{J. J. Sylvester1} obtained $g(A)=a_1a_2-a_1-a_2$ in 1882. In the general case $n\geq 3$, it is known that $g(A)$ can not be given by closed formulas of a certain type \cite{F.Curtis}. For $n=3$, both G. Denham \cite{G.Denham} and A. Tripathi \cite{A. Tripathi1} have studied formula $g(A)$. Moreover, many formulas for special cases have been determined (see \cite{M. Hujter,A. Brauer,Roberts1,E. S. Selmer,A. Tripathi2,A. Tripathi3,A. L. Dulmage,T.Komatsu2022Arx,T.Komatsu22Arx,AMRobles}).

This work is along the line of the combinatorial approach for Frobenius numbers developed in \cite{Liu-Xin}.
By using this method and the Four-Square Theorem in number theory, we were able to solve an open problem proposed in \cite{D. Einstein} by D. Einstein, D. Lichtblau, A. Strzebonski and S. Wagon. See \cite{Fliuxin22}. The open problem is to characterize $g(A)$ for the square sequence $A=(a, a+B)=(a, a+1^2, a+2^2, ..., a+k^2)$.
In order to solve this open problem, we introduced a generating function to establish certain bounds, which allows us to analyze to obtain a full characterization of
$g(A)$ as a congruence class function modulo $k^2$ when $a$ is sufficiently large.

We observe that the argument in \cite{Fliuxin22} naturally extends for general $B=(1,b_2,\dots, b_k)$. Similarly we can characterize $g(A)$ where $A=(a,ha+dB)=(a, ha+db_1, ha+db_2,...,ha+db_k)$, $\gcd(a,d)=1, b_1=1$. The resulting formula is a congruence class function modulo $b_k$ when $a$ is sufficiently large.
The bound for $a$ might be large for general sequence $B$, but is good for a special class of $B$, called \emph{orderly sequence} as we shall introduce next.

We need to introduce some basic definitions and results in \cite{J.C.Rosales,Liu-Xin}. Throughout this paper, $\mathbb{Z}$, $\mathbb{N}$, and $\mathbb{P}$ denote the set of all integers, non-negative integers, and positive integers, respectively.

A subset $S$ is a \emph{submonoid} of $\mathbb{N}$ if $S\subseteq \mathbb{N}$, $0\in S$ and $S$ is closed under the sum in $\mathbb{N}$. If $\mathbb{N}\setminus S$ is finite, then we say that $S$ is a \emph{numerical semigroup} (see \cite{A.Assi,J.C.Rosales}). For the sequence $A=(a, b_1, b_2,..., b_k)$ with $\gcd(A)=1$, $a<b_1<b_2\cdots <b_k$, $a,b_i\in \mathbb{P}$, the set $\mathcal{R}=\left\{ ax+\sum_{i=1}^kb_ix_i\ \mid x, x_i\in \mathbb{N}\right\}$ is a numerical semigroup \cite{J.C.Rosales}. We denote by
$$Ape(A,a)=\{N_0,N_1,N_2,...,N_{a-1}\},$$
the \emph{Ap\'ery set} of $A$, where $N_r:=\min\{ a_0\mid a_0\equiv r\mod a, \ a_0\in \mathcal{R}\}$, $0\leq r\leq a-1$.

The following fundamental result of A. Brauer and J. E. Shockley are widely used.
\begin{thm}[\cite{J. E. Shockley}]
Suppose $A:=(a, B)=(a, b_1, b_2, ..., b_k)$, $\gcd(A)=1$, $d\in \mathbb{P}$, $\gcd(a,d)=1$. Then the Frobenius number is
$$g(A)=\max Ape(A,a)-a=\mathop{\max}\limits_{r\in \{ 0, 1, ..., a-1\}} \{ N_r\} -a =\mathop{\max}\limits_{r\in \{ 0, 1, ..., a-1\}} \{ N_{dr}\} -a.$$
\end{thm}

Suppose $A=(a, ha+dB)=(a, ha+db_1, ..., ha+db_k)$. The computation of
$N_{dr}$ is reduced to a minimization problem defined by
\begin{equation}\label{hahaha4}
O_B(M):=\min\left\{\sum_{i=1}^kx_i \mid \sum_{i=1}^kb_ix_i=M, \ \ M,x_i\in \mathbb{N}, 1\leq i\leq k\right\}.
\end{equation}

\begin{lem}[\cite{Liu-Xin}]\label{0202}
Let $A=(a, ha+db_1, ..., ha+db_k)$,  $k, h, d, b_i \in\mathbb{P}$ and $\gcd(A)=1$,  $\gcd(a, d)=1$. For a given $ 0\leq r\leq a-1$,  we have
\begin{equation}\label{0203}
N_{dr}=\min \{N_{dr}(m) \mid m\in\mathbb{N}\}, \qquad \text{where}\ \ \ N_{dr}(m):=O_B(ma+r) \cdot ha+(ma+r)d.
\end{equation}
\end{lem}

Now, let $B=(b_1, b_2, ..., b_k)$, $1=b_1<b_2<\cdots b_k$. For given $B$ and $M\in \mathbb{N}$, determining $O_B(M)$ is called \emph{the change-making problem} \cite{AnnAdamaszek}. We always assume $b_1=1$ to avoid the absence of the solution. The greedy strategy is to use as many of the maximum as possible, than as many of the next one as possible, and so on. We denote $G_B(M)$ the number of elements used in $B$ by the greedy strategy. So we have $O_B(M)\leq G_B(M)$. If the greedy solution is always optimal, i.e., $O_B(M)=G_B(M)$ for all $M$, then we call the sequence $B$ \emph{orderly}; Otherwise, we call sequence $B$ \emph{non-orderly}. We know that $B=(1)$, $B=(1,b_2)$ are orderly sequences. The sequence $B=(1,6,13)$ is non-orderly. Because $18=13+1+1+1+1+1=6+6+6$, we have $3=O_B(18)<G_B(18)=6$.  A representation of $M$ by exactly $O_B(M)$ elements is said to be \emph{optimal}.

In this paper, we characterize the formula $g(A)$ for $A=(a,ha+dB)=(a, ha+d, ha+db_2,...,ha+db_k)$, $\gcd(a,d)=1$. We assume $k\geq 2$ since the $k=1$ case is known.
We find that the ``Stable" property for $O_B(M)$ not only holds for the square sequence, but also holds for general $B$. This leads to a characterization of the Frobenius formula $g(A)$ as a ``congruence class function" modulo $b_k$, i.e., $g(A(a))$ is divided into $b_k$ classes according to $a \mod b_k$. Moreover, each segment is a quadratic polynomial in $a$ with leading coefficient $\frac{h}{b_k}$. This result holds when $a$ is greater than a certain bound. If $B$ is an orderly sequence, we obtain a good bound.

The paper is organized as follows.
In Section 2, we mainly discuss the ``Stable" property of $O_B(M)$ for general $B$.
In Section 3, we characterize the Frobenius formula $g(A(a))=g(a,ha+dB)$, where $B=(1,b_2,...,b_k)$. As special cases, we obtain the Frobenius formula $g(a,ha+dB)$ for $B=(1,2,b,b+1)$.
Section 4 focus on orderly sequences  $B$. We obtain good bounds for the Frobenius formula $g(A(a))$. Furthermore, calculate $g(a,ha+dB)$ for $B=(1,2,b,b+1,2b)$, $B=(1,b,2b-1)$ and $B=(1,2,...,k,K)$.
In Section 5, we discuss the case $b_1> 1$.

\section{The ``Stable" Property}\label{stablepropcc}

Let $A=(a, ha+dB)=(a,ha+db_1,ha+db_2,...,ha+db_k)$ and $b_1=1$.
In this section we mainly focus on characterizing the ``Stable" property about $O_B(M)$. Our starting point is Lemma \ref{0202}, which says that  we need to compute the quantity
$$N_{dr}=\min\left\{O_B(ma+r) \cdot ha+(ma+r)d \mid \sum_{i=1}^kx_ib_i=ma+r, \ \ m,x_i\in \mathbb{N}, 1\leq i\leq k\right\}$$
for each $r$.

In summary, we need to study how to calculate $O_B(M)$ efficiently.
By Equation \eqref{hahaha4}, it is natural to consider the generating function:
$$ F(t, q) =\prod_{i=1}^k \frac{1}{1-t q^{b_i}} = \sum_{n\ge 0}  \left( \sum_{ x_1+b_2x_2+\cdots+b_kx_k=n} t^{x_1+x_2+\cdots +x_k} \right)   q^n .$$

If we call a solution $(x_1,x_2,...,x_k)$ for \eqref{hahaha4} satisfying $x_1+\cdots +x_k=O_B(M)$ \emph{optimal}, then
the generating function $f(t,q) =\sum_{n\ge 0} t^{O_B(n)} q^n$ extracts only one optimal representation, weighted by $t^{O_B(n)}$,
 for each $n$.
It is easy to see that
$$f(t,q) =\sum_{n\ge 0} t^{O_B(n)} q^n := \circledast F(t,q),$$
where $\circledast$ is the operator defined as follows.

\begin{dfn}\label{hahaha12}
For a power series $G(t,q)$ in $t,q$ with nonnegative coefficients, define
$\circledast G(t,q)$ be the power series obtained from $G(t,q)$ by picking a minimum degree (in $t$) term in each coefficient (in $q$).
\end{dfn}

The basic fact $\circledast F(t,q) G(t,q) = \circledast (\circledast F(t,q)) (\circledast G(t,q))$ allows us to use Maple to compute the first $M+1$ terms of $f(t,q)$ quickly for any given reasonably large $M$:
\begin{enumerate}
  \item Start with $f_1:=\sum_{n= 0}^M t^n q^n$.

  \item Suppose $f_{i-1}$ has been obtained. Then $f_i$ is obtained by: first compute $\circledast f_{i-1} \cdot \sum_{n=0}^{h_i} t^n q^{b_i n}$
  where $h_i= \lfloor M/b_i \rfloor$ will be optimized,
  and then remove all terms with degree in $q$ larger than $M$.

  \item Set $f(t,q)|_{q^{\leq M}}=f_k$.
\end{enumerate}
Thus we have proved the following result.
\begin{lem}\label{hahaha10}
Let $B=(1,b_2,...,b_k)$, $b_i\in\mathbb{N}$, $2\leq i\leq k$. For a given $M$, the first $M+1$ term $f(t,q)|_{q^{\leq M}}$ of $f(t,q)$ can be computed in polynomial time in $M$. Consequently, $O_B(r), \ r\leq M$ can be computed in polynomial time in $M$.
\end{lem}

We first give an example to facilitate our understanding of the subsequent results. The following instance says that $O_B(r)$ is stable in some sense.
\begin{exa}\label{b1613exm}
Let $B=(1,11,14)$. For a given reasonably large $M=140$, we get the following results by Maple. To see the stable property clearly, we divide by $f(t,q)=\sum_{i=0}^{13} f^i$, where $f^i$ extract all terms corresponding to  $q^{14s+i}$. We also bold faced all terms not implied by the stable property.
\begin{small}
\begin{align*}
f^0&=\mathbf{1}+tq^{14}+{t}^{2}{q}^{28}+{t}^{3}{q}^{42}+{t}^{4}{q}^{56}+{t}^{5}{q}^{70}
+t^{6}q^{84}+t^7q^{98}+t^{8}q^{112}+t^9q^{126}+\cdots
\\f^1&=\mathbf{tq}+t^2q^{15}+{t}^{3}{q}^{29}+{t}^{4}{q}^{43}+{t}^{5}{q}^{57}
+{t}^{6}{q}^{71}+t^{7}q^{85}+t^8q^{99}+t^9q^{113}+t^{10}q^{127}+\cdots
\\f^2&=\mathbf{t^2q^2+t^3q^{16}+{t}^{4}{q}^{30}+t^4q^{44}}+{t}^{5}{q}^{58}
+{t}^{6}{q}^{72}+{t}^{7}{q}^{86}+t^{8}q^{100}+t^9q^{114}+t^{10}q^{128}+\cdots
\\f^3&=\mathbf{t^3q^3+t^4q^{17}+{t}^{5}{q}^{31}+t^5q^{45}}+{t}^{6}{q}^{59}
+{t}^{7}{q}^{73}+{t}^{8}{q}^{87}+t^9q^{101}+t^{10}q^{115}+t^{11}q^{129}+\cdots
\\f^4&=\mathbf{t^4q^4+t^5q^{18}+{t}^{6}{q}^{32}+{t}^{6}{q}^{46}+{t}^{7}{q}^{60}
+{t}^{8}{q}^{74}+t^8q^{88}}+t^{9}q^{102}+t^{10}q^{116}+t^{11}q^{130}+\cdots
\\f^5&=\mathbf{t^5q^5+t^6q^{19}+{t}^{3}{q}^{33}}+{t}^{4}{q}^{47}+{t}^{5}{q}^{61}
+t^6q^{75}+{t}^{7}{q}^{89}+t^8q^{103}+t^9q^{117}+t^{10}q^{131}+\cdots
\\f^6&=\mathbf{t^6q^6+t^7q^{20}+{t}^{4}{q}^{34}}+{t}^{5}{q}^{48}+{t}^{6}{q}^{62}
+{t}^{7}{q}^{76}+t^{8}q^{90}+t^{9}q^{104}+t^{10}q^{118}+t^{11}q^{132}+\cdots
\\f^7&=\mathbf{t^7q^7+t^8q^{21}+{t}^{5}{q}^{35}+{t}^{6}{q}^{49}+{t}^{7}{q}^{63}
+{t}^{7}{q}^{77}}+t^8q^{91}+t^9q^{105}+t^{10}q^{119}+t^{11}q^{133}+\cdots
\\f^8&=\mathbf{t^8q^8+t^2q^{22}}+{t}^{3}{q}^{36}+{t}^{4}{q}^{50}+{t}^{5}{q}^{64}
+{t}^{6}{q}^{78}+t^7q^{92}+t^8q^{106}+t^9q^{120}+t^{10}q^{134}+\cdots
\\f^9&=\mathbf{t^9q^9+t^3q^{23}}+{t}^{4}{q}^{37}+{t}^{5}{q}^{51}+{t}^{6}{q}^{65}
+{t}^{7}{q}^{79}+t^8q^{93}+t^9q^{107}+t^{10}q^{121}+t^{11}q^{135}+\cdots
\\f^{10}&=\mathbf{t^{10}q^{10}+t^4q^{24}+{t}^{5}{q}^{38}+{t}^{6}{q}^{52}+{t}^{6}{q}^{66}}
+{t}^{7}{q}^{80}+t^8q^{94}+t^9q^{108}+t^{10}q^{122}+t^{11}q^{136}+\cdots
\\f^{11}&=\mathbf{tq^{11}}+t^2q^{25}+{t}^{3}{q}^{39}+{t}^{4}{q}^{53}+{t}^{5}{q}^{67}
+{t}^{6}{q}^{81}+t^7q^{95}+t^8q^{109}+t^9q^{123}+t^{10}q^{137}+\cdots
\\f^{12}&=\mathbf{t^2q^{12}}+t^3q^{26}+{t}^{4}{q}^{40}+{t}^{5}{q}^{54}+{t}^{6}{q}^{68}
+{t}^{7}{q}^{82}+t^8q^{96}+t^9q^{110}+t^{10}q^{124}+t^{11}q^{138}+\cdots
\\f^{13}&=\mathbf{t^3q^{13}+t^4q^{27}+{t}^{5}{q}^{41}+{t}^{5}{q}^{55}}+{t}^{6}{q}^{69}
+{t}^{7}{q}^{83}+t^8q^{97}+t^9q^{111}+t^{10}q^{125}+t^{11}q^{139}+\cdots.
\end{align*}
\end{small}
Any $O_B(r)$ can be deduced from the bold faced terms. For instance, $O_B(130)=O_B(88+3\cdot 14)=O_B(88)+3=11$.
\end{exa}

Let $c$ be the maximum value of the degree of $t$ in the first $b_k$ term of $f(t,q)$, i.e., $c=\max\{O_B(r) \mid 0\leq r\leq b_k-1\}$. Clearly $c\geq 1$.
\begin{lem}\label{hahaha1}
Let $B=(1,b_2,...,b_k)$, $b_i\in\mathbb{N}$, $2\leq i\leq k$. If $r\in \mathbb{N}$, then
$$\left\lceil \frac{r}{b_k}\right\rceil \leq O_B(r)\leq \left\lfloor \frac{r}{b_k}\right\rfloor+c.$$
\end{lem}
\begin{proof}
Let $r=sb_k+r_1$, $s\geq 0$, $0\leq r_1\leq b_k-1$. If $r_1=0$, then
clearly $O_B(r)=s$. If $1\le r_1\le b_k-1$,
$$\left\lceil \frac{r}{b_k}\right\rceil= s+1\leq O_B(r)\leq s+ O_B(r_1)\leq s+ c=\left\lfloor \frac{r}{b_k}\right\rfloor+c.$$
\end{proof}

The following lemma says that $O_B(r)$ is ``Stable", and hence reduce the computation of $O_B(r)$ to finitely many $r$'s.

\begin{lem}\label{hahaha17}
Let $B=(1,b_2,...,b_k)$, $b_i\in\mathbb{N}$, $2\leq i\leq k$. For any $r\geq \left\lceil\frac{(c-1)b_{k-1}}{b_k-b_{k-1}}\right\rceil\cdot b_k$, $r\in \mathbb{P}$, we have $$O_B(b_k+r)=O_B(r)+1.$$
\end{lem}
\begin{proof}
Obviously $O_B(b_k+r)\leq O_B(r)+1$ always hold. If $b_k+r$ has a minimal representation including a $b_k$, then
removing the $b_k$ leads to the other side inequality $O_B(r)\leq O_B(b_k+r)-1$ and hence $O_B(b_k+r)=O_B(r)+1$.

Thus assume to the contrary that $O_B(b_k+r)<O_B(r)+1$. Then $b_k+r$ cannot have a minimal representation using $b_k$.
We prove that this is impossible under our assumption.

Let $r=mb_k+j$, where $0\leq j\leq b_k-1$. By definition of $c$, we have $j=\sum_{i=1}^c e_i$, $e_i\in \{0,1,b_2,...,b_{k-1}\}$ for $1\leq i\leq c$. By $r\geq \left\lceil\frac{(c-1)b_{k-1}}{b_k-b_{k-1}}\right\rceil\cdot b_k$, we have $m\geq \left\lceil\frac{(c-1)b_{k-1}}{b_k-b_{k-1}}\right\rceil$. The $m$ is chosen so that $mb_k\geq (m+c-1)b_{k-1}$. For such $m$, the following formula holds:
$$r=mb_k+j=mb_k+\sum_{i=1}^c e_i \geq (m+c-1)b_{k-1}.$$

This implies that without using $b_k$, we have $O_B(b_k+r)\geq (m+c-1)+2$. Then we have $m+c+1\leq O_B(b_k+r)< O_B(r)+1\leq m+c+1$, a contradiction.
\end{proof}

\begin{thm}\label{t-OBM}
Given a sequence $B=(1,b_2,...,b_k)$, there is a polynomial time algorithm in $b_k$ for computing $O_B(M)$ for all $M\in \P$.
\end{thm}
\begin{proof}
By Lemma \ref{hahaha17},For any $r\geq \left\lceil\frac{(c-1)b_{k-1}}{b_k-b_{k-1}}\right\rceil\cdot b_k$, we have $O_B(sb_k+r)=s+O_B(r)$ for $s\in \mathbb{N}$.
Thus we only need to calculate $O_B(r)$ for $0\leq r \leq \left\lceil\frac{(c-1)b_{k-1}}{b_k-b_{k-1}}\right\rceil\cdot b_k$. By Lemma \ref{hahaha10}, $O_B(r),\ r\le M$ can be computed in polynomial time in $M$. This completes the proof.
\end{proof}

\section{The Frobenius Formula $g(A(a))$}

We will characterize the Frobenius formula $g(A)$, where $A=(a,ha+dB)=(a,ha+d,ha+db_2,...,ha+db_k)$. In this section, we always assume that $c\geq 2$. In Section \ref{aspecial}, we will discuss the case $c=1$. We need some results about $N_{dr}$.

\begin{lem}\label{hahaha19}
Given a sequence $B=(1,b_2,...,b_k)$.
If $a\geq (c-1)b_k$, then for a given $r$, we have $N_{dr}=N_{dr}(0)$.
\end{lem}
\begin{proof}
For a given $r$ and any $m\geq 0$, let $ma+r=sb_k+r_1$, where $0\leq r_1<b_k$. By $a\geq (c-1)b_k$, we have $(m+1)a+r=sb_k+a+r_1\geq (s+c-1)b_k+r_1$. If $r_1\neq 0$, then we have $O_B((m+1)a+r)\geq (s+c)$. By Lemma \ref{hahaha1}, we have $O_B(ma+r)=s+l$ for a certain $1\leq l \leq c$. Therefore we have
\begin{align*}
N_{dr}(m+1)\geq (s+c)ha+((m+1)a+r)d \geq (s+l)ha+(ma+r)d=N_{dr}(m).
\end{align*}
If $r_1=0$, then we have $O_B((m+1)a+r)\geq (s+c-1)$ and $O_B(ma+r)=s$. Similarly, we have $N_{dr}(m+1)\geq N_{dr}(m)$.

Then $N_r(m)$ is increasing, and hence minimizes at $m=0$.
\end{proof}

Now, we show that $N_{dr}$ behave nicely when $a$ is large, by the means of we have the following result.
\begin{lem}\label{hahaha20}
Given a sequence $B=(1,b_2,...,b_k)$. Let $u=\max\left\{c-1, \left\lceil\frac{(c-1)b_{k-1}}{b_k-b_{k-1}}\right\rceil\right\}$. For any $a\geq (u+c-1)b_k$, we have
$$\mathop{\max}\limits_{0\leq r\leq a-1}\{N_{dr}\}=\mathop{\max}\limits_{a-b_k\leq r\leq a-1}\{N_{dr}\}.$$
\end{lem}
\begin{proof}
By Lemma \ref{hahaha19} and $a\geq (u+c-1)b_k$, we have $N_r=N_r(0)$. By Lemma \ref{hahaha17}, obviously for a given $r\geq \left\lceil\frac{(c-1)b_{k-1}}{b_k-b_{k-1}}\right\rceil\cdot b_k$, $O_B(sb_k+r)$ is increasing with $s\geq 0$.

Now, we consider the following case. If $y=mb_k+\sum_{i=1}^ce_i$ where $m\geq 0$, and $e_i\in \{0,1,b_2,...,b_{k-1}\}$ are not all $0$ for $1\leq i\leq c$. We have $O_B(y)\leq m+c$, $O_B(y+b_k)\geq m+2$, ..., $O_B(y+(c-1)b_k)\geq m+c$. So for a given $r\geq ub_k+(c-2)b_k$, we must have $O_B(r)\geq O_B(r-(c-1)b_k)$. More generally, we have $O_B(r)\geq O_B(r-sb_k), s\geq 0$. Therefore when $a\geq (u+c-2)b_k+b_k$, we have $\mathop{\max}\limits_{0\leq r\leq a-1}\{N_{dr}\}=\mathop{\max}\limits_{a-b_k\leq r\leq a-1}\{N_{dr}\}$.
\end{proof}

For a given sequence $B=(1,b_2,...,b_k)$, the following theorem will show that the Frobenius formula $g(A)$ with $A=(a, ha+dB)$ is a ``congruence class function" with $b_k$ classes.
\begin{thm}\label{hahaha21}
Let $A(a)=(a,ha+dB)=(a, ha+d, ha+db_2, ..., ha+db_k)$, $\gcd(a,d)=1$, $u=\max\left\{c-1, \left\lceil\frac{(c-1)b_{k-1}}{b_k-b_{k-1}}\right\rceil\right\}$ and $h\geq \left\lceil \frac{d}{u+c-1}\right\rceil$, $a,h,d \in \mathbb{P}$. There exist two nonnegative integer sequences $W_B=(w_{j})_{0\le j\le b_k-1}$ and $R_B=(r_{j})_{0\le j \le b_k-1}$, such that for all $a\geq (u+c-1)b_k$, we have
$$g(A(a))=(w_{j}h-1)a+r_{j}d+(ha+b_kd)\left(\left\lfloor \frac{a}{b_k}\right\rfloor -u-c+1\right),\ \ \textrm{where } a\equiv j\mod b_k.$$
Moreover, $W_B$ and $R_B$ are both increasing, and $w_{b_k-1}-w_{0}\le 1$.
\end{thm}
The proof is parallel to the proof of \cite[Theorem 4.11]{Fliuxin22}, but we need details to obtain the bounds.
\begin{proof}
Let $a=sb_k+j\geq (u+c-1)b_k$, where $s\geq (u+c-1)$ and $0\leq j\leq b_k-1$. By Lemma \ref{hahaha20}, we have $\mathop{\max}\limits_{0\leq r\leq a-1}\{N_{dr}\}=\mathop{\max}\limits_{a-b_k\leq r\leq a-1}\{N_{dr}\}$.

By $h\geq \left\lceil \frac{d}{u+c-1}\right\rceil$ and $a\geq (u+c-1)b_k$, we have
$h\geq \lceil \frac{b_kd}{a}\rceil$. Furthermore, we have $ha-b_kd\geq 0$.
Therefore $N_{dr}=O_B(r)\cdot ha+rd$ is dominated by the coefficient $O_B(r)$.
This suggest that we shall first find
$\iota=\mathop{\max}\limits_{a-b_k\leq r\leq a-1} O_B(r)$
and then find the largest $a-b_k\leq \widehat{r} \leq a-1$ satisfying $O_B(\widehat{r})=\iota$. Then $\max \{N_{dr}\} =O_B(\widehat{r})\cdot ha +\widehat{r}d$, and the Frobenius number is $O_B(\widehat{r})\cdot ha +\widehat{r}d-a$.

Now, let $(u+c-1)b_k\leq a\leq (u+c)b_k-1$ and $a\equiv j\mod b_k$. So $j$ can take all the numbers in $\{0,1,2,...,b_k-1\}$. For a certain $j$, according to the above discussion, there is an $a-b_k\leq r_{j}\leq a-1$ such that we can get
$g(A)=\max\{N_{dr}\}-a=w_{j}\cdot ha+r_{j}d-a$,
where $w_{j}, r_{j}\in \mathbb{N}$ and $w_{j}=O_B(r_{j})$.

For the same $j$, if $a=sb_k+j$, $s>(u+c-1)$, then $a=(s-u-c+1)b_k+(u+c-1)b_k+j$. By Lemma \ref{hahaha9}, we have $\overline{r}=(s-u-c+1)b_k+r_{j}$ such that $\max\{N_{dr}\}=N_{d\overline{r}}$. Then
\begin{align*}
g(A(a))&=(s-u-c+1+w_j)ha+((s-u-c+1)b_k+r_j)d-a
\\&=(w_jh-1)a+r_jd+(ha+db_k)\cdot \left(\left\lfloor \frac{a}{b_k}\right\rfloor-u-c+1\right).
\end{align*}
Therefore for any $0\leq j\leq b_k-1$, all $w_{j}$ and $ r_{j}$ form the sequences  $W_B$ and $R_B$ respectively.

By the calculation process of $w_{j}$ and $r_{j}$,  it is easy to see that $(u+c-2)b_k\leq r_{j}< (u+c)b_k$. When $a=(u+c-1)b_k$, let $\widehat{r}$ be the largest number such that
$$(u+c-2)b_k\leq \widehat{r}\leq (u+c-1)b_k-1\ \ \ \ \ \text{and}\ \ \ \ \   O_B(\widehat{r})=\mathop{\max}\limits_{(u+c-2)b_k\leq r\leq (u+c-1)b_k-1} O_B(r) .$$
In the next period $(u+c-1)b_k$ to $(u+c)b_k-1$ corresponding to $a=(u+c)b_k$, $\widehat{r}^{\prime}=\widehat{r}+b_k$ satisfies
$$(u+c-1)b_k\leq \widehat{r}^{\prime}\leq (u+c)b_k-1\ \ \ \ \ \text{and}\ \ \ \ \ O_B(\widehat{r}^{\prime})=O_B(\widehat{r})+1=\mathop{\max}\limits_{(u+c-1)b_k\leq r\leq (u+c)b_k-1} O_B(r) .$$
Therefore, in any interval of length $b_k$ starting with $a\equiv j \mod b_k$ that lies between $(u+c-2)b_k$ and $(u+c)b_k-1$, we have two cases: i) If the interval contains $\widehat{r}^{\prime}$, then
$r_{j}=\widehat{r}^{\prime}$; ii) If the interval contains $\widehat{r}$, then $r_{j}$ lies between $\widehat{r}$
and $\widehat{r}^{\prime}$ and $O_B(r_{j})$ is either
$O_B(\widehat{r})$ or $O_B(\widehat{r})+1$. Moreover,
$r_{j}$ is increasing with respect to $j$. Obviously, from the above discussion, we can know that $W_B$ is increasing and $w_{b_k-1}-w_{0}\le 1$.
\end{proof}

\begin{exa}\label{examplenonordely}
Let $B=(1,11,14)$. For $A(a)=(ha+dB)=(a, ha+d,ha+11d,ha+14d)$, by Example \ref{b1613exm}, we have $c=10$, $u=33$, $h\geq \lceil \frac{d}{42}\rceil$ and $a\geq 588$. We know $574\leq r_{j}< 601$.
\begin{align*}
\circledast F(t,q)= \cdots &+{t}^{41}{q}^{574}+{t}^{42}{q}^{575}+{t}^{42}{
q}^{576}+t^{43}q^{577}+{t}^{43}{q}^{578}+{t}^{42}{q}^{579}+{t}^{43}{q}^{580}
\\&+{t}^{43}{q}^{581}+{t}^{42}{q}^{582}+{t}^{43}{q}^{583}
+t^{43}q^{584}+t^{42}q^{585}+t^{43}q^{586}+\underline{t^{43}q^{587}}
\\&+{t}^{42}{q}^{588}+\mathbf{{t}^{43}{q}^{589}}+\mathbf{{t}^{43}{q}^{590}}+\mathbf{{t}^{44}{q}^{591}}
+\mathbf{{t}^{44}{q}^{592}}+{t}^{43}{q}^{593}+\mathbf{{t}^{44}{q}^{594}}
\\&+\mathbf{{t}^{44}{q}^{595}}+{t}^{43}{q}^{596}+\mathbf{t^{44}q^{597}}+\mathbf{t^{44}q^{598}}
+t^{43}q^{599}+\mathbf{t^{44}q^{600}}+\underline{t^{44}q^{601}}
\\&+\cdots,
\end{align*}
where we only displayed the necessary part in our proof. Multiplying the first (second) row by $tq^{14}$ gives the third (fourth) row.

The terms for $\widehat{r}$ and $\widehat{r}'$ are underlined in the two rows with
$$O_B(\widehat{r})=\mathop{\max}\limits_{574\leq r\leq 587} O_B(r)=O_B(587)=43  \quad \text{ and } \quad
O_B(\widehat{r}^{\prime})=44.$$
We discuss how to find $r_{j}$ from the above formula.

For $a=588$, $r_{0}=587$ is just $\widehat{r}$ and consequently
$w_{0}=O_B(587)=43$, $g(A)=43ha+587d-a$. Furthermore we conclude that if $a\equiv 0 \mod 14$ and $a\geq588$,
then
$$g(A(a))=(43h-1)a+587d+(ha+14d)\left(\left\lfloor \frac{a}{14}\right\rfloor -42\right).$$

For $a=589$, since $42=O_B(588)<O_B(587)=43$, we have $r_{1}=587$.
For $a=590$, the situation is different and we have $r_{2}=43$. This is because we meet the bold faced term, which corresponds to $43=O_B(589)\ge O_B(587)=43$. The bold faced terms are left-to-right maximums with respect to the power of $t$.
By a similar reasoning, we obtain
\begin{align*}
&W_B=[43,43,43,43,44,44,44,44,44,44,44,44,44,44],\\ &R_B=[587,587,589,590,591,592,592,594,595,595,597,598,598,600],
\end{align*}
from which we can build the following Frobenius formula for $B=(1,11,14)$:
$$\begin{aligned}
  g(A(a))=
\left\{
    \begin{array}{lc}
         (43h-1)a+587d+(ha+14d)(\lfloor \frac{a}{14}\rfloor -42) &\ \text{if}\ \ a\equiv 0,1 \mod 14;\ \ \ \ \\
         (43h-1)a+589d+(ha+14d)(\lfloor \frac{a}{14}\rfloor -42) &\ \text{if}\ \ a\equiv 2 \mod 14;\ \ \ \ \ \ \ \\
         (43h-1)a+590d+(ha+14d)(\lfloor \frac{a}{14}\rfloor -42) &\ \text{if}\ \ a\equiv 3 \mod 14;\ \ \ \ \ \ \ \\
         (44h-1)a+591d+(ha+14d)(\lfloor \frac{a}{14}\rfloor -42) &\ \text{if}\ \ a\equiv 4 \mod 14;\ \ \ \ \ \ \ \\
         (44h-1)a+592d+(ha+14d)(\lfloor \frac{a}{14}\rfloor -42) &\ \text{if}\ \ a\equiv 5,6 \mod 14;\ \ \ \ \\
         (44h-1)a+594d+(ha+14d)(\lfloor \frac{a}{14}\rfloor -42) &\ \text{if}\ \ a\equiv 7 \mod 14;\ \ \ \ \ \ \ \\
         (44h-1)a+595d+(ha+14d)(\lfloor \frac{a}{14}\rfloor -42) &\ \text{if}\ \ a\equiv 8,9 \mod 14;\ \ \ \ \\
         (44h-1)a+597d+(ha+14d)(\lfloor \frac{a}{14}\rfloor -42) &\ \text{if}\ \ a\equiv 10 \mod 14;\ \ \ \ \ \\
         (44h-1)a+598d+(ha+14d)(\lfloor \frac{a}{14}\rfloor -42) &\ \text{if}\ \ a\equiv 11,12 \mod 14;\\
         (44h-1)a+600d+(ha+14d)(\lfloor \frac{a}{14}\rfloor -42) &\ \text{if}\ \ a\equiv 13 \mod 14;\ \ \ \ \ \\
    \end{array}
\right.
\end{aligned}$$
\end{exa}

\begin{exa}\label{example141117}
For $A(a)=(ha+dB)=(a, ha+d,ha+4d,ha+11d,ha+17d)$, similarly, we have $c=4$, $u=6$, $h\geq \lceil \frac{d}{9}\rceil$ and $a\geq 153$. Therefore we obtain
$$\begin{aligned}
  g(A(a))=
\left\{
    \begin{array}{lc}
         (11h-1)a+150d+(ha+17d)(\lfloor \frac{a}{17}\rfloor -9) &\ \text{if}\ \ a\equiv 0,1,2 \mod 17;\ \ \ \ \ \ \ \ \ \ \ \ \ \ \ \ \\
         (11h-1)a+155d+(ha+17d)(\lfloor \frac{a}{17}\rfloor -9) & \text{if}\ \ a\equiv 3 \mod 17;\ \ \ \ \ \ \ \ \ \ \ \ \ \ \ \ \ \ \ \ \\
         (11h-1)a+156d+(ha+17d)(\lfloor \frac{a}{17}\rfloor -9) & \text{if}\ \ a\equiv 4,5,6 \mod 17;\ \ \ \ \ \ \ \ \ \ \ \ \ \ \ \\
         (11h-1)a+159d+(ha+17d)(\lfloor \frac{a}{17}\rfloor -9) & \text{if}\ \ a\equiv 7 \mod 17;\ \ \ \ \ \ \ \ \ \ \ \ \ \ \ \ \ \ \ \ \\
         (12h-1)a+160d+(ha+17d)(\lfloor \frac{a}{17}\rfloor -9) & \text{if}\ \ a\equiv 8,9,10,11,12,13 \mod 17;\\
         (12h-1)a+166d+(ha+17d)(\lfloor \frac{a}{17}\rfloor -9) & \text{if}\ \ a\equiv 14 \mod 17;\ \ \ \ \ \ \ \ \ \ \ \ \ \ \ \ \ \ \\
         (12h-1)a+167d+(ha+17d)(\lfloor \frac{a}{17}\rfloor -9) & \text{if}\ \ a\equiv 15,16 \mod 17.\ \ \ \ \ \ \ \ \ \ \ \ \ \
    \end{array}
\right.
\end{aligned}$$
\end{exa}

\begin{rem}
The bound $a\geq (u+c-1)b_k$ in Theorem \ref{hahaha21} is not tight. In fact, when $a\geq 116$, the formula in Example \ref{examplenonordely} holds. If $h=d=1$ and $2\leq a\leq 115$, the formula in Example \ref{examplenonordely} only holds at $a=2,102$. Similarly, the formula in Example \ref{example141117} holds for $a\geq 42$. If $h=d=1$ and $2\leq a\leq 41$, the formula in Example \ref{example141117} only holds at $a=3,8,25,26,27,31$.

In the proof of Theorem \ref{hahaha21}, by calculating $O_B(r)$, we provide a general method for computing the coefficients $w_{j}, r_{j}$. Finally, an exact Frobenius formula can be obtained.
\end{rem}

\begin{thm}
Let $A(a)=(a,ha+dB)=(a, ha+d, ha+2d, ha+bd, ha+(b+1)d)$, $\gcd(a,d)=1$, $b\geq 4$ and $a,h,d,b \in \mathbb{P}$.
If $h\geq \left\lceil \frac{2d}{(b-3)(b+1)}\right\rceil$, $a\geq \frac{(b-3)(b+1)^2}{2}$ for odd $b$, or if $h\geq \left\lceil \frac{2d}{(b-2)(b+1)}\right\rceil$, $a\geq \frac{(b-2)(b+1)^2}{2}$ for even $b$, then we have
$$\begin{aligned}
  g(A(a))=
\left\{
    \begin{array}{lc}
    \lfloor\frac{a}{b+1}\rfloor(ha+(b+1)d)+(j-1)d-a &\ \text{if}\ \ a\equiv j \in \{0,1\} \mod b+1;\ \ \ \ \ \ \ \\
    \lfloor\frac{a}{b+1}\rfloor(ha+(b+1)d)+(j-1)d+(h-1)a &\ \text{if}\ \ a\equiv j \in \{2,3,...,b\} \mod b+1.
    \end{array}
\right.
\end{aligned}$$
\end{thm}
\begin{proof}
The proof is similar to Example \ref{b1613exm}. If $b$ is odd, then we have
\begin{small}
\begin{align*}
f^0&=1+tq^{b+1}+{t}^{2}{q}^{2b+2}\cdots+{t}^{b-3}{q}^{b^2-2b-3}+{t}^{b-2}{q}^{b^2-b-2}
+\cdots+{t}^{\frac{(b-3)(b+1)}{2}}{q}^{\frac{(b-3)(b+1)^2}{2}}+\cdots
\\f^1&=tq+t^2q^{b+2}+{t}^{3}{q}^{2b+3}\cdots+{t}^{b-2}{q}^{b^2-2b-2}+{t}^{b-1}{q}^{b^2-b-1}
+\cdots+{t}^{\frac{(b-3)(b+1)}{2}+1}{q}^{\frac{(b-3)(b+1)^2}{2}+1}+\cdots
\\f^2&=tq^2+t^2q^{b+3}+{t}^{3}{q}^{2b+4}\cdots+{t}^{b-2}{q}^{b^2-2b-1}+{t}^{b-1}{q}^{b^2-b}
+\cdots+{t}^{\frac{(b-3)(b+1)}{2}+1}{q}^{\frac{(b-3)(b+1)^2}{2}+2}+\cdots
\\f^3&=t^2q^3+t^3q^{b+4}+{t}^{4}{q}^{2b+5}\cdots+{t}^{b-2}{q}^{b^2-2b}+{t}^{b-1}{q}^{b^2-b+1}
+\cdots+{t}^{\frac{(b-3)(b+1)}{2}+1}{q}^{\frac{(b-3)(b+1)^2}{2}+3}+\cdots
\\f^4&=t^2q^4+t^3q^{b+5}+{t}^{4}{q}^{2b+6}\cdots+{t}^{b-2}{q}^{b^2-2b+1}+{t}^{b-1}{q}^{b^2-b+2}
+\cdots+{t}^{\frac{(b-3)(b+1)}{2}+1}{q}^{\frac{(b-3)(b+1)^2}{2}+4}+\cdots
\\&\cdots
\\&\cdots
\\f^{b-3}&=t^{\frac{b-3}{2}}q^{b-3}+t^{\frac{b+1}{2}}q^{2b-2}+{t}^{\frac{b+3}{2}}{q}^{3b-1}\cdots
+{t}^{b-2}{q}^{b^2-b-6}+{t}^{b-1}{q}^{b^2-5}+\cdots
+{t}^{\frac{(b-3)(b+1)}{2}+1}{q}^{\frac{(b-3)(b+1)^2}{2}+b-3}+\cdots
\\f^{b-2}&=t^{\frac{b-1}{2}}q^{b-2}+t^{\frac{b+1}{2}}q^{2b-1}+{t}^{3}{q}^{3b}\cdots
+{t}^{b-2}{q}^{b^2-b-5}+{t}^{b-1}{q}^{b^2-4}+\cdots
+{t}^{\frac{(b-3)(b+1)}{2}+1}{q}^{\frac{(b-3)(b+1)^2}{2}+b-2}+\cdots
\\f^{b-1}&=t^{\frac{b-1}{2}}q^{b-1}+t^{2}q^{2b}+{t}^{3}{q}^{3b+1}\cdots
+{t}^{b-2}{q}^{b^2-b-4}+{t}^{b-1}{q}^{b^2-3}+\cdots
+{t}^{\frac{(b-3)(b+1)}{2}+1}{q}^{\frac{(b-3)(b+1)^2}{2}+b-1}+\cdots
\\f^{b}&=tq^{b}+t^{2}q^{2b+1}+{t}^{3}{q}^{3b+2}\cdots
+{t}^{b-2}{q}^{b^2-b-3}+{t}^{b-1}{q}^{b^2-2}+\cdots
+{t}^{\frac{(b-3)(b+1)}{2}+1}{q}^{\frac{(b-3)(b+1)^2}{2}+b}+\cdots.
\end{align*}
\end{small}
Now, we have $c=\frac{b-1}{2}$ and $u=\frac{b(b-3)}{2}$. Similarly, if $b$ is even, then we have $c=\frac{b}{2}$ and $u=\frac{b(b-2)}{2}$. By Theorem \ref{hahaha21}, we can get the Frobenius formula $g(A(a))$.
\end{proof}

\section{A Special Case: Orderly Sequence Formula}\label{aspecial}

In this section, let $B$ be an orderly sequence, i.e, $O_B(M)=G_B(M)$, where $G_B(M)$ the number of elements used in $B$ by greedy strategy. Now, for the $A=(a,ha+dB)=(a,ha+d,ha+db_2,...,ha+d_k)$, we will obtain a friendly lower bound for $a$ in Theorem \ref{hahaha21}. For related research on orderly sequence, see \cite{AnnAdamaszek,LJCowen,TCHu,MJMagazine}.

\subsection{The Frobenius Formula for Orderly Sequence $B$}
The ``Stable" property for orderly sequence is obvious.
\begin{lem}\label{hahaha9}
Let $B=(1,b_2,...,b_k)$ be an orderly sequence. For any $r\in \mathbb{P}$, we have $$O_B(b_k+r)=O_B(r)+1.$$
\end{lem}

Similarly, Lemma \ref{hahaha1} and Lemma \ref{hahaha19} is still correct. For Lemma \ref{hahaha20}, we briefly describe as follows. Let $c$ be the maximum value of the degree of $t$ in the first $b_k$ term of $f(t,q)$, i.e., $c=\max\{O_B(r) \mid 0\leq r\leq b_k-1\}$.

\begin{lem}\label{hahaha11}
Suppose $B=(1,b_2,...,b_k)$ is orderly. For any $a\geq (c-1)b_k$, $c\geq 2$, we have
$$\mathop{\max}\limits_{0\leq r\leq a-1}\{N_{dr}\}=\mathop{\max}\limits_{a-b_k\leq r\leq a-1}\{N_{dr}\}.$$
\end{lem}
\begin{proof}
By Lemma \ref{hahaha19} and $a\geq (c-1)b_k$, we have $N_{dr}=N_{dr}(0)=O_B(r)\cdot ha+rd$. By Lemma \ref{hahaha9}, for a given $r$, $O_B(sb_k+r)$ is increasing with respect to $s\geq 0$. Therefore when $a\geq (c-1)b_k$, we have $\mathop{\max}\limits_{0\leq r\leq a-1}\{N_{dr}\}=\mathop{\max}\limits_{a-b_k\leq r\leq a-1}\{N_{dr}\}$.
\end{proof}

\begin{thm}\label{hahaha8}
Given an orderly sequence $B=(1,b_2,...,b_k)$, $c\geq 2$. Let $A(a)=(a,ha+dB)=(a, ha+d, ha+db_2, ..., ha+db_k)$, $\gcd(a,d)=1$ and $h\geq \left\lceil \frac{d}{c-1}\right\rceil$, $a,h,d \in \mathbb{P}$. There exist two nonnegative integer sequences $W_B=(w_{j})_{0\le j\le b_k-1}$ and $R_B=(r_{j})_{0\le j \le b_k-1}$, such that for all $a\geq (c-1)b_k$, we have
$$g(A(a))=(w_{j}h-1)a+r_{j}d+(ha+b_kd)\left(\left\lfloor \frac{a}{b_k}\right\rfloor -c+1\right),\ \ \textrm{where } a\equiv j\mod b_k.$$
Moreover, $W_B$ and $R_B$ are both increasing, and $w_{b_k-1}-w_{0}\le 1$.
\end{thm}

\begin{exa}\label{hahaha22}
Let $B=(1,5,9)$ and $A(a)=(a, ha+d, ha+5d, ha+9d)$. We have
\begin{align*}
f(t,q)=&(1+tq+t^2q^2+t^3q^3+t^4q^4+tq^5+t^2q^6+t^3q^7+t^4q^8)\cdot \frac{1}{1-tq^9}
\\&=\left(\frac{1-(tq)^5}{1-tq}+\frac{tq^5(1-(tq)^4)}{1-tq}\right)\cdot \frac{1}{1-tq^9}
\\&=\frac{1-(tq)^5+tq^5(1-(tq)^4)}{(1-tq)(1-tq^9)}.
\end{align*}
So $h\geq \left\lceil \frac{d}{3}\right\rceil$. The lower bound is $a\geq 3b_k=27$. Now, we have
$$\begin{aligned}
  g(A(a))=
\left\{
    \begin{array}{lc}
         (6h-1)a+26d+(ha+9d)(\lfloor \frac{a}{9}\rfloor -3) &\ \text{if}\ \ a\equiv 0,1,2,3 \mod 9;\\
         (6h-1)a+30d+(ha+9d)(\lfloor \frac{a}{9}\rfloor -3) & \text{if}\ \ a\equiv 4 \mod 9;\ \ \ \ \ \ \ \ \\
         (7h-1)a+31d+(ha+9d)(\lfloor \frac{a}{9}\rfloor -3) & \text{if}\ \ a\equiv 5,6,7,8 \mod 9.
    \end{array}
\right.
\end{aligned}$$
\end{exa}

\begin{rem}
The bound $a\geq (c-1)b_k$ in Lemma \ref{hahaha11} and Theorem \ref{hahaha8} are friendly. In fact, when $a\geq 24$, the formula in Example \ref{hahaha22} holds.
\end{rem}

\subsection{Some Interesting Formulas}

In \cite[Proposition 4.2]{AnnAdamaszek}, the sequence $B=(1,2,b,b+1,2b)$, $b\geq 4$ is orderly. A. Adamaszek and M. Adamaszek have discussed the sequence $B$'s properties in detail in \cite{AnnAdamaszek}. Note that the partial sequence $(1,2,b,b+1), b\geq 4$ of $B$ is not orderly.

\begin{thm}
Let $A(a)=(a,ha+dB)=(a, ha+d, ha+2d, ha+bd, ha+(b+1)d, ha+2bd)$, $\gcd(a,d)=1$, $b\geq 4$ and $a,h,d,b \in \mathbb{P}$. If $b$ is even and $h\geq \left\lceil \frac{2d}{b-2}\right\rceil$, $a\geq b(b-2)$, then we have
$$\begin{aligned}
  g(A(a))=
\left\{
    \begin{array}{lc}
    \left(\lfloor\frac{a}{2b}\rfloor+\frac{b}{2}-1\right)ha+2bd\lfloor\frac{a}{2b}\rfloor-a-d &\ \text{if}\ \ a\equiv 0,1,2,...,b-3 \mod 2b;\ \ \ \ \\
    \left(\lfloor\frac{a}{2b}\rfloor+\frac{b}{2}-1\right)ha+2bd\lfloor\frac{a}{2b}\rfloor-a+bd-3d &\ \text{if}\ \ a\equiv b-2 \mod 2b;\ \ \ \ \ \ \ \ \ \ \ \ \ \ \ \ \\
    \left(\lfloor\frac{a}{2b}\rfloor+\frac{b}{2}-1\right)ha+2bd\lfloor\frac{a}{2b}\rfloor-a+bd-2d &\ \text{if}\ \ a\equiv b-1 \mod 2b;\ \ \ \ \ \ \ \ \ \ \ \ \ \ \ \ \\
    \left(\lfloor\frac{a}{2b}\rfloor+\frac{b}{2}\right)ha+2bd\lfloor\frac{a}{2b}\rfloor-a+bd-d &\ \text{if}\ \ a\equiv b,b+1,...,2b-2 \mod 2b;\\
    \left(\lfloor\frac{a}{2b}\rfloor+\frac{b}{2}\right)ha+2bd\lfloor\frac{a}{2b}\rfloor-a+2bd-2d &\ \text{if}\ \ a\equiv 2b-1 \mod 2b.\ \ \ \ \ \ \ \ \ \ \ \ \ \
    \end{array}
\right.
\end{aligned}$$
If $b$ is odd and $h\geq \left\lceil \frac{2d}{b-1}\right\rceil$, $a\geq b(b-1)$, then we have
$$\begin{aligned}
  g(A(a))=
\left\{
    \begin{array}{lc}
    \left(\lfloor\frac{a}{2b}\rfloor+\frac{b-1}{2}\right)ha+2bd\lfloor\frac{a}{2b}\rfloor-a-d &\ \text{if}\ \ a\equiv 0,1,2,...,b-2 \mod 2b;\ \ \ \ \\
    \left(\lfloor\frac{a}{2b}\rfloor+\frac{b-1}{2}\right)ha+2bd\lfloor\frac{a}{2b}\rfloor-a+bd-2d &\ \text{if}\ \ a\equiv b-1 \mod 2b;\ \ \ \ \ \ \ \ \ \ \ \ \ \ \ \ \\
    \left(\lfloor\frac{a}{2b}\rfloor+\frac{b-1}{2}\right)ha+2bd\lfloor\frac{a}{2b}\rfloor-a+bd-d &\ \text{if}\ \ a\equiv b,b+1,...,2b-3 \mod 2b;\\
    \left(\lfloor\frac{a}{2b}\rfloor+\frac{b-1}{2}\right)ha+2bd\lfloor\frac{a}{2b}\rfloor-a+2bd-3d &\ \text{if}\ \ a\equiv 2b-2 \mod 2b;\ \ \ \ \ \ \ \ \ \ \ \ \ \ \ \\
    \left(\lfloor\frac{a}{2b}\rfloor+\frac{b-1}{2}\right)ha+2bd\lfloor\frac{a}{2b}\rfloor-a+2bd-2d &\ \text{if}\ \ a\equiv 2b-1 \mod 2b.\ \ \ \ \ \ \ \ \ \ \ \ \ \ \
    \end{array}
\right.
\end{aligned}$$
\end{thm}
\begin{proof}
If $b$ is even, we have
\begin{align*}
f(t,q)=&\Big(1+tq+tq^2+t^2q^3+t^2q^4+\cdots+t^{\frac{b-2}{2}}q^{b-3}
+t^{\frac{b-2}{2}}q^{b-2}+t^{\frac{b}{2}}q^{b-1}
\\&+tq^b+tq^{b+1}+t^2q^{b+2}+t^2q^{b+3}+\cdots t^{\frac{b}{2}}q^{2b-2}+t^{\frac{b}{2}}q^{2b-1}\Big)\cdot \frac{1}{1-tq^{2b}}.
\end{align*}
So we have $c=\frac{b}{2}$ and $a\geq (c-1)2b=b(b-2)$. Suppose $a\equiv j \mod 2b$, $0\leq j\leq 2b-1$. In $f(t,q)$, the coefficient before in $q^{b(b-2)-1}=q^{(\frac{b}{2}-2)\cdot 2b+2b-1}$ is $t^{\frac{b}{2}-2+\frac{b}{2}}=t^{b-2}$. Therefore we have $w_{j}=b-2$ and $r_j=b^2-2b-1$ for $0\leq j\leq b-3$. Similarly, the coefficient before in $q^{(\frac{b}{2}-1)\cdot 2b+b-3}$ is $t^{b-2}$ in $f(t,q)$. So we have $w_{b-2}=b-2$ and $r_{b-2}=b^2-b-3$. Again because there is $t^{b-2}q^{(\frac{b}{2}-1)2b+b-2}$ in $f(t,q)$, we have $w_{b-1}=b-2$ and $r_{b-1}=b^2-b-2$. For $b\leq j\leq 2b-2$, there is $t^{b-1}q^{(\frac{b}{2}-1)2b+b-1}$ in $f(t,q)$. We have $w_{j}=b-1$ and $r_{j}=b^2-b-1$ for $b\leq j\leq 2b-2$. Finally, because there is $t^{b-1}q^{(\frac{b}{2}-1)2b+2b-1}$ in $f(t,q)$, we have $w_{2b-1}=b-1$ and $r_{2b-1}=b^2-2$. So we obtain the sequences $W_B$ and $R_B$ in Theorem \ref{hahaha8}. By the formula in Theorem \ref{hahaha8}, we obtain the Frobenius formula $g(A(a))$ in which $b$ is even.

If $b$ is odd, we have
\begin{align*}
f(t,q)=&\Big(1+tq+tq^2+t^2q^3+t^2q^4+\cdots+t^{\frac{b-1}{2}}q^{b-2}
+t^{\frac{b-1}{2}}q^{b-1}
\\&+tq^b+tq^{b+1}+t^2q^{b+2}+t^2q^{b+3}+\cdots t^{\frac{b-1}{2}}q^{2b-3}+t^{\frac{b-1}{2}}q^{2b-2}+t^{\frac{b+1}{2}q^{2b-1}}\Big)\cdot \frac{1}{1-tq^{2b}}.
\end{align*}
The calculation process of Frobenius formula $g(A(a))$ is similar to that $b$ is even. This completes the proof.
\end{proof}

Obviously, the sequence $(1,b)$, $b\geq 2$ is orderly. By the \emph{One-Point Theorem} \cite{LJCowen,TCHu,MJMagazine}, we know that $B=(1,b,2b-1), b\geq 3$ is an orderly sequence.
\begin{thm}
Let $A(a)=(a,ha+dB)=(a, ha+d, ha+bd, ha+(2b-1)d)$, $\gcd(a,d)=1$, $b\geq 3$ and $a,h,d,b \in \mathbb{P}$. Suppose $h\geq \left\lceil \frac{d}{b-2}\right\rceil$, $a\geq 2b^2-5b+2$, then we have
\begin{small}
$$\begin{aligned}
  g(A(a))=
\left\{
    \begin{array}{lc}
    \left(\lfloor\frac{a}{2b-1}\rfloor+b-2\right)ha+(2b-1)d\lfloor\frac{a}{2b-1}\rfloor-a-d &\ \text{if}\ \ a\equiv 0,1,...,b-2 \mod 2b-1;\ \ \ \\
    \left(\lfloor\frac{a}{2b-1}\rfloor+b-2\right)ha+(2b-1)d\lfloor\frac{a}{2b-1}\rfloor-a+(b-2)d &\ \text{if}\ \ a\equiv b-1 \mod 2b-1;\ \ \ \ \ \ \ \ \ \ \ \ \\
    \left(\lfloor\frac{a}{2b-1}\rfloor+b-1\right)ha+(2b-1)d\lfloor\frac{a}{2b-1}\rfloor-a+(b-1)d &\ \text{if}\ \ a\equiv b,...,2b-2 \mod 2b-1.\ \ \ \
    \end{array}
\right.
\end{aligned}$$
\end{small}
\end{thm}
\begin{proof}
By the sequence $b=(1,b,2b-1)$, $b\geq 3$ is orderly, we have
\begin{align*}
f(t,q)=&\Big(1+tq+t^2q^2+t^3q^3+\cdots+t^{b-3}q^{b-3}+t^{b-2}q^{b-2}+t^{b-1}q^{b-1}
\\&+tq^b+t^2q^{b+1}+t^3q^{b+2}+\cdots t^{b-3}q^{2b-4}+t^{b-2}q^{2b-3}+t^{b-1}q^{2b-2}\Big)\cdot \frac{1}{1-tq^{2b-1}}.
\end{align*}
By Theorem \ref{hahaha8}, we have $c=b-1$, $h\geq \lceil\frac{d}{b-2}\rceil$ and $a\geq 2b^2-5b+2$. The proof of remaining parts are routine.
\end{proof}

In \cite{E. S. Selmer}, E. S. Selmer studies the sequence $A=(a, a+1, a+2, ..., a+k, a+K)$ with $K>k$. In \cite{Rodseth}, \"O. J. R\"odseth further obtains $g(A)$ for $A=(a,a+d,..., a+kd, a+Kd)$ with an additional term.  In \cite{T.Komatsu20220323}, there are also some results about \emph{Sylvester sum} for$A=(a,a+d,..., a+kd, a+Kd)$. Now, we obtain the following results by Theorem \ref{hahaha8}.
\begin{thm}
Let $A(a)=(a,ha+dB)=(a, ha+d, ha+2d,..., ha+kd, ha+Kd)$, $\gcd(a,d)=1$, $K-1>k\geq 2$ and $a,h,d,k,K \in \mathbb{P}$. Let $K=sk+t$, $0\leq t\leq k-1$ and $s\geq 1$. For $2\leq t\leq k-1$, if $h\geq \left\lceil \frac{d}{s}\right\rceil$, $a\geq sK$, then we have
\begin{small}
$$\begin{aligned}
  g(A(a))=
\left\{
    \begin{array}{lc}
    \left(\lfloor\frac{a}{K}\rfloor+s\right)ha+Kd\lfloor\frac{a}{K}\rfloor-a-d &\ \text{if}\ \ a\equiv j\in\{ 0,1,...,(s-1)k+1\} \mod K;\ \ \ \ \ \ \\
    \left(\lfloor\frac{a}{K}\rfloor+s\right)ha+Kd\lfloor\frac{a}{K}\rfloor-a+(j-1)d &\ \text{if}\ \ a\equiv j\in\{ (s-1)k+2,...,sk+1\} \mod K;\ \ \ \\
    \left(\lfloor\frac{a}{K}\rfloor+s+1\right)ha+Kd\lfloor\frac{a}{K}\rfloor-a+(j-1)d &\ \text{if}\ \ a\equiv j\in\{ sk+2,...,K-1\} \mod K.\ \ \ \ \ \ \ \ \ \ \
    \end{array}
\right.
\end{aligned}$$
\end{small}
For $t=0,1$, if $h\geq \left\lceil \frac{d}{s-1}\right\rceil$, $a\geq (s-1)K$, then we have
\begin{small}
$$\begin{aligned}
  g(A(a))=
\left\{
    \begin{array}{lc}
    \left(\lfloor\frac{a}{K}\rfloor+s-1\right)ha+Kd\lfloor\frac{a}{K}\rfloor-a-d &\ \text{if}\ \ a\equiv j\in\{ 0,1,...,(s-2)k+1\} \mod K;\ \ \ \ \ \ \ \ \ \ \ \ \ \ \\
    \left(\lfloor\frac{a}{K}\rfloor+s-1\right)ha+Kd\lfloor\frac{a}{K}\rfloor-a+(j-1)d &\ \text{if}\ \ a\equiv j\in\{(s-2)k+2,...,(s-1)k+1\} \mod K;\ \ \ \ \\
    \left(\lfloor\frac{a}{K}\rfloor+s\right)ha+Kd\lfloor\frac{a}{K}\rfloor-a+(j-1)d &\ \text{if}\ \ a\equiv j\in\{(s-1)k+2,...,K-1\} \mod K.\ \ \ \ \ \ \ \ \ \ \
    \end{array}
\right.
\end{aligned}$$
\end{small}
\end{thm}
\begin{proof}
If $2\leq t\leq k-1$, we have
\begin{align*}
f(t,q)=&\Big(1+\big(tq+tq^2+\cdots+tq^k\big)+\big(t^{2}q^{k+1}+t^{2}q^{k+2}\cdots+t^{2}q^{2k}\big)
+\cdots +\big(t^sq^{(s-1)k+1}
\\&+t^sq^{(s-1)k+2}+\cdots t^{s}q^{sk}\big)+\big(t^{s+1}q^{sk+1}+t^{s+1}q^{sk+2}+\cdots +t^{s+1}q^{sk+t-1}\big)\Big)\cdot \frac{1}{1-tq^{K-1}}.
\end{align*}
By Theorem \ref{hahaha8}, we have $c=s+1$, $h\geq \lceil\frac{d}{s}\rceil$ and $a\geq sK$.

If $t=0,1$, by $K>k+1$, we have $s\geq 2$. Therefore we have
\begin{align*}
f(t,q)=&\Big(1+\big(tq+tq^2+\cdots+tq^k\big)+\big(t^{2}q^{k+1}+t^{2}q^{k+2}\cdots+t^{2}q^{2k}\big)
+\cdots +\big(t^{s-1}q^{(s-2)k+1}
\\&+t^{s-1}q^{(s-2)k+2}+\cdots t^{s-1}q^{(s-1)k}\big)+\big(t^{s}q^{(s-1)k+1}+\cdots +t^{s}q^{sk+t-1}\big)\Big)\cdot \frac{1}{1-tq^{K-1}}.
\end{align*}
By Theorem \ref{hahaha8}, we have $c=s$, $h\geq \lceil\frac{d}{s-1}\rceil$ and $a\geq (s-1)K$.

The proof of remaining parts are routine.
\end{proof}

\subsection{The case $c=1$}
Now, suppose $c=1$. We know that $B=(1,2,...,k)$, $k\geq 2$. This is an orderly sequence. Lemma \ref{hahaha1}, Lemma \ref{hahaha19} and Lemma \ref{hahaha9} is still correct. We can force the assumption that $c=2$. Then Lemma \ref{hahaha11} and Theorem \ref{hahaha8} is also holds. Moreover, there is no restriction between $h$ and $d$. Therefore, we can obtain the following corollary.

\begin{cor}[\cite{E. S. Selmer}]\label{cor1234k}
Let $A(a)=(a,ha+dB)=(a, ha+d, ha+2d, ..., ha+kd)$, $\gcd(a,d)=1$, $a\geq 2$, $a,h,d,k \in \mathbb{P}$. Then we have
\begin{equation}\label{corcis1}
g(A(a))=ha\left(\left\lfloor\frac{a-2}{k}\right\rfloor+1\right)+(d-1)(a-1)-1.
\end{equation}
\end{cor}
\begin{proof}
We first suppose $a\geq k$, i.e., we force that $c=2$. By Theorem \ref{hahaha8}, we have
$$\begin{aligned}
  w_j=
\left\{
    \begin{array}{lc}
         1 &\ \text{if}\ \ j= 0,1;\ \ \ \ \ \ \ \ \ \\
         2 &\ \text{if}\ \ 2\leq j\leq k-1.
    \end{array}
\right.
\end{aligned}$$
and
$$r_j=k+j-1,\ \ 0\leq j\leq k-1.$$
Therefore we have
$$\begin{aligned}
  g(A(a))=
\left\{
    \begin{array}{lc}
         (h-1)a+(k+j-1)d+(ha+kd)(\lfloor \frac{a}{k}\rfloor -1) &\ \text{if}\ \  j=0,1\ \ \ \ \ \ \ \ ;\\
         (2h-1)a+(k+j-1)d+(ha+kd)(\lfloor \frac{a}{k}\rfloor -1) &\ \text{if}\ \ 2\leq j\leq k-1;
    \end{array}
\right.
\end{aligned}$$
where $a\equiv j \mod k$. Furthermore, the above formula also holds for $2\leq a\leq k-1$. It is easy to verify that the above formula is equivalent to Equation \eqref{corcis1}.
\end{proof}

For Corollary \ref{cor1234k},
A. Brauer \cite{A. Brauer} first computed $g(A(a))=g(a,a+1,...,a+k)$. J. B. Roberts \cite{Roberts1} extended this result to $A=(a,a+d,...,a+kd)$. E. S. Selmer \cite{E. S. Selmer} further generalized to the case $A=(a, ha+d, ha+2d, ..., ha+kd)$ (also see \cite{A. Tripathi2}).

\section{The Case $b_1\neq 1$}

We believe the ideas developed here also works for general $B=(b_1,b_2,\dots, b_k)$, i.e., without the restriction $b_1=1$. Indeed, we may assume $\gcd(b_1,b_2,\dots,b_k)=1$ because of the $d$. We claim that $O_B(M)$ also have the ``Stable" property. The proof is similar but much more complicated.
We only illustrate the idea by an example.

Let us consider the ``Stable" property of $B=(4,8,15,17)$. Similar to Example \ref{b1613exm}, we divide by $f(t,q)=\sum_{i=0}^{16} f^i$, where $f^i$ extract all terms corresponding to $q^{17s+i}, s\geq 0$. Again we bold faced all terms not implied by the stable property. We added some underlined $0\cdot q^*$ terms for clarity. For instance, $0\cdot q^{18}$ in $f^1$ means that $18$ is not representable by $B$.
\begin{small}
\begin{align*}
f^0&=\mathbf{1}+tq^{17}+{t}^{2}{q}^{34}+{t}^{3}{q}^{51}+{t}^{4}{q}^{68}+{t}^{5}{q}^{85}
+t^{6}q^{102}+t^7q^{119}+t^{8}q^{136}+t^9q^{153}+\cdots
\\f^1&=\mathbf{\underline{0\cdot q+0\cdot q^{18}}+{t}^{4}{q}^{35}+{t}^{5}{q}^{52}+{t}^{6}{q}^{69}
+{t}^{7}{q}^{86}+t^{8}q^{103}+t^8q^{120}}+t^9q^{137}+t^{10}q^{154}+\cdots
\\f^2&=\mathbf{\underline{0\cdot q^2}+t^2q^{19}}+{t}^{3}{q}^{36}+t^4q^{53}+{t}^{5}{q}^{70}
+{t}^{6}{q}^{87}+{t}^{7}{q}^{104}+t^{8}q^{121}+t^9q^{138}+t^{10}q^{155}+\cdots
\\f^3&=\mathbf{\underline{0\cdot q^3}+t^3q^{20}+{t}^{4}{q}^{37}+t^5q^{54}+{t}^{6}{q}^{71}
+{t}^{7}{q}^{88}+{t}^{7}{q}^{105}}+t^8q^{122}+t^{9}q^{139}+t^{10}q^{156}+\cdots
\\f^4&=\mathbf{tq^4}+t^2q^{21}+{t}^{3}{q}^{38}+{t}^{4}{q}^{55}+{t}^{5}{q}^{72}
+{t}^{6}{q}^{89}+t^7q^{106}+t^{8}q^{123}+t^{9}q^{140}+t^{10}q^{157}+\cdots
\\f^5&=\mathbf{\underline{0\cdot q^5+0\cdot q^{22}}+{t}^{4}{q}^{39}+{t}^{5}{q}^{56}+{t}^{6}{q}^{73}
+t^6q^{90}}+{t}^{7}{q}^{107}+t^8q^{124}+t^9q^{141}+t^{10}q^{158}+\cdots
\\f^6&=\mathbf{\underline{0\cdot q^6}+t^2q^{23}}+{t}^{3}{q}^{40}+{t}^{4}{q}^{57}+{t}^{5}{q}^{74}
+{t}^{6}{q}^{91}+t^{7}q^{108}+t^{8}q^{125}+t^{9}q^{142}+t^{10}q^{159}+\cdots
\\f^7&=\mathbf{\underline{0\cdot q^7}+t^3q^{24}+{t}^{4}{q}^{41}+{t}^{5}{q}^{58}+{t}^{5}{q}^{75}}
+{t}^{6}{q}^{92}+t^7q^{109}+t^8q^{126}+t^{9}q^{143}+t^{10}q^{160}+\cdots
\\f^8&=\mathbf{tq^8}+t^2q^{25}+{t}^{3}{q}^{42}+{t}^{4}{q}^{59}+{t}^{5}{q}^{76}
+{t}^{6}{q}^{93}+t^7q^{110}+t^8q^{127}+t^9q^{144}+t^{10}q^{161}+\cdots
\\f^9&=\mathbf{\underline{0\cdot q^9+0\cdot q^{26}}+{t}^{5}{q}^{43}+{t}^{4}{q}^{60}}+{t}^{5}{q}^{77}
+{t}^{6}{q}^{94}+t^7q^{111}+t^8q^{128}+t^{9}q^{145}+t^{10}q^{162}+\cdots
\\f^{10}&=\mathbf{\underline{0\cdot q^{10}}+t^3q^{27}}+{t}^{4}{q}^{44}+{t}^{5}{q}^{61}+{t}^{6}{q}^{78}
+{t}^{7}{q}^{95}+t^8q^{112}+t^9q^{129}+t^{10}q^{146}+t^{11}q^{163}+\cdots
\\f^{11}&=\mathbf{\underline{0\cdot q^{11}}+t^4q^{28}+{t}^{3}{q}^{45}}+{t}^{4}{q}^{62}+{t}^{5}{q}^{79}
+{t}^{6}{q}^{96}+t^7q^{113}+t^8q^{130}+t^9q^{147}+t^{10}q^{164}+\cdots
\\f^{12}&=\mathbf{t^2q^{12}}+t^3q^{29}+{t}^{4}{q}^{46}+{t}^{5}{q}^{63}+{t}^{6}{q}^{80}
+{t}^{7}{q}^{97}+t^8q^{114}+t^9q^{131}+t^{10}q^{148}+t^{11}q^{165}+\cdots
\\f^{13}&=\mathbf{\underline{0\cdot q^{13}}+t^2q^{30}}+{t}^{3}{q}^{47}+{t}^{4}{q}^{64}+{t}^{5}{q}^{81}
+{t}^{6}{q}^{98}+t^7q^{115}+t^8q^{132}+t^{9}q^{149}+t^{10}q^{166}+\cdots
\\f^{14}&=\mathbf{\underline{0\cdot q^{14}}+t^3q^{31}}+{t}^{4}{q}^{48}+{t}^{5}{q}^{65}+{t}^{6}{q}^{82}
+{t}^{7}{q}^{99}+t^8q^{116}+t^9q^{133}+t^{10}q^{150}+t^{11}q^{167}+\cdots
\\f^{15}&=\mathbf{tq^{15}}+t^2q^{32}+{t}^{3}{q}^{49}+{t}^{4}{q}^{66}+{t}^{5}{q}^{83}
+{t}^{6}{q}^{100}+t^7q^{117}+t^8q^{134}+t^{9}q^{151}+t^{10}q^{168}+\cdots
\\f^{16}&=\mathbf{t^2q^{16}}+t^3q^{33}+{t}^{4}{q}^{50}+{t}^{5}{q}^{67}+{t}^{6}{q}^{84}
+{t}^{7}{q}^{101}+t^8q^{118}+t^9q^{135}+t^{10}q^{152}+t^{11}q^{169}+\cdots.
\end{align*}
\end{small}
Similarly, any $O_B(r)$ can be deduced from the bold faced terms, but numbers in $\mathcal{NR}(B)=\{1,2,3,5,6,7,9,10,11,13,14,18,22,26\}$
are not representable by $B$.

\begin{thm}\label{a481517}
Let $A(a)=(a,ha+dB)=(a, ha+4d, ha+8d, ha+15d, ha+17d)$, $\gcd(a,d)=1$ and $a,h,d \in \mathbb{P}$.
If $a\geq 90$, $h\geq \left\lceil \frac{17d}{90}\right\rceil$ and $a\equiv j \mod 17$, $0\leq j\leq 16$, then we have
$$g(A(a))=(\gamma_j h-1)a+(26+j)d+(ha+17d)\lfloor\frac{a}{17}\rfloor,$$
where $(\gamma_0,\gamma_1,...,\gamma_{17})=(2,3,2,3,2,3,2,3,2,3,3,3,3,3,3,3,3)$.
\end{thm}
\begin{proof}
By the ``Stable" property of $B=(4,8,15,17)$, the $O_B(r)$ are determined by $(O_B(r))_{r\leq 121}$. For convenience, we first consider the case $a\geq 136=8\times 17$.
Let $\mathcal{R}(B)=\{r \mid 0\leq r\leq a-1\}\setminus \mathcal{NR}(B)$.

By Equation \eqref{0203}, we know that $N_{dr}(m):=O_B(ma+r) \cdot ha+(ma+r)d$.
For $a\geq 136=8\times 17$ and $r\in \mathcal{R}(B)$, we have
\begin{align*}
N_{dr}(m+1)-N_{dr}(m)=(O_B((m+1)a+r)-O_B(ma+r))ha+ad\geq 0.
\end{align*}
So we have $N_{dr}=N_{dr}(0)$ for $r\in \mathcal{R}(B)$.
For $r\in \mathcal{NR}(B)$, one can see that $N_{dr}(m)$ only exists for $m\ge 1$. Therefore we have $N_{dr}=N_{dr}(1)$ in this case.

By $h\geq \left\lceil \frac{17d}{90}\right\rceil$, we have $90 h\geq 17d$ and $ha\geq 17d$. This implies that $N_{dr}$ is dominated by $O_B(r)ha$ or $O_B(a+r)ha$.
This suggests the following computation of $g(A(a))$. We begin with the case $a=17s$, $s\ge 8$.
The base step is when $a=136=17\times 8$. We have
\begin{align*}
\max_{0\leq r\leq a-1}\{N_{dr}\}&=\max\left\{\max_{r\in \mathcal{R}(B)}\{N_{dr}\} , \max_{r\in \mathcal{NR}(B)}\{N_{dr}\}\right\}
\\&=\max\left\{\max_{119\leq r\leq 135}\{O_B(r)\cdot ha+rd\} , \max_{r\in \mathcal{NR}(B)}\{O_B(a+r)\cdot ha+(a+r)d\}\right\}
\\&=\max\left\{\{9ha+135d\} , \{O_B(136+26)\cdot ha+(136+26)d \}\right\}
\\&=10ha+162d.
\end{align*}
For $s\geq 8$, we write $a=(s-8)\cdot 17+136$. Therefore
\begin{align*}
g(A(a))&=(s-8+10)ha+((s-8)\cdot 17+(136+26))d-a
\\&=(2h-1)a+26d+(ha+17d)\left\lfloor\frac{a}{17}\right\rfloor.
\end{align*}

The computation of $g(A(a))$ where $a=17s+j$, $s\geq 8$, $1\leq j\leq 16$, can be done similarly. Readers can check that the above formula also holds for $90\leq a\leq 135$.
\end{proof}

For $a=89$, the formula fails because the condition $\max_{0\leq r\leq a-1}\{N_{dr}\}=O_B(89+26)\cdot ha+(89+26)d$ is not satisfied:
$O_B(89+14)=8>O_B(89+26)=7$, and we have $\max_{r\in \mathcal{NR}}\{N_{dr}\}=O_B(89+14)\cdot ha+(89+14)d$.

For a specific $B$ with $\gcd(B)=1$ and $b_1>1$, the idea in the proof of Theorem \ref{a481517} still applies to compute the Frobenius formula $g(A(a))$. But the process is much more complicated.
For example, we need to modify the $c$ as in Lemma \ref{hahaha1}, which is the maximum value of the degree of $t$ in the first $b_k$ term of $f(t,q)$.
This has to be modified. We may let $c$ be the maximum value of the degree of $t$ in the $\left\lceil\frac{g(B)}{b_k}\right\rceil$-th column of $f(t,q)=\sum_{i=0}^{b_k-1} f^i$, where $g(B)$ is the Frobenius number for the sequence $B$. In other words, we have
$$c=\max\left\{O_B(r) \mid \left\lceil\frac{g(B)}{b_k}\right\rceil \cdot b_k\leq r\leq \left(\left\lceil\frac{g(B)}{b_k}\right\rceil+1\right)b_k-1\right\}.$$

Now one can see that the following corollary should be correct.
\begin{cor}\label{b1b2b3b4bk}
For a given positive integer sequence $B=(b_1,b_2,...,b_k)$ satisfying $b_1<b_2<\cdots <b_k$ and $\gcd(B)=1$. Let $A(a)=(a,ha+dB)=(a, ha+b_1d, ha+b_2d, ..., ha+b_kd)$, $\gcd(a,d)=1$ and $a,h,d\in\mathbb{P}$. When $a$ is greater than a certain bound, the Frobenius formula $g(A(a))$ is a ``congruence class function" modulo $b_k$, and each segment is a quadratic polynomial in $a$ with leading coefficient $\frac{h}{b_k}$.
\end{cor}

\section{Concluding Remark}

Our main results are Theorems \ref{hahaha21}, \ref{hahaha8} and Corollary \ref{b1b2b3b4bk}. We give a characterization of the Frobenius formula for
$A(a)=(a,ha+dB)=(a,ha+d, ha+db_2,..., ha+db_k)$ for a general sequence $B=(1,b_2, b_3, ..., b_k)$. This indeed gives a polynomial time algorithm in $b_k$
for the computation of $g(A(a))$. In some special cases, we can even compute $g(A(a))$ for symbolic $B$.

\noindent
{\small \textbf{Acknowledgements:}
This work was partially supported by the National Natural Science Foundation of China [12071311].

\end{document}